\documentclass[10pt]{amsart}
\usepackage{amssymb}
\usepackage{MnSymbol,times}
\usepackage{amsthm,amsmath}
\usepackage{cite}

\title{Barycentrically associative and preassociative functions}

\author{Jean-Luc Marichal}
\address{Mathematics Research Unit, FSTC, University of Luxembourg \\
6, rue Coudenhove-Kalergi, L-1359 Luxembourg, Luxembourg} \email{jean-luc.marichal[at]uni.lu }

\author{Bruno Teheux}
\address{Mathematics Research Unit, FSTC, University of Luxembourg \\
6, rue Coudenhove-Kalergi, L-1359 Luxembourg, Luxembourg} \email{bruno.teheux[at]uni.lu }

\date{March 28, 2015}

\begin{document}

\theoremstyle{plain}
\newtheorem{theorem}{Theorem}[section]
\newtheorem{lemma}[theorem]{Lemma}
\newtheorem{proposition}[theorem]{Proposition}
\newtheorem{corollary}[theorem]{Corollary}
\newtheorem{fact}[theorem]{Fact}
\newtheorem*{main}{Main Theorem}

\theoremstyle{definition}
\newtheorem{definition}[theorem]{Definition}
\newtheorem{example}[theorem]{Example}

\theoremstyle{remark}
\newtheorem*{conjecture}{Conjecture}
\newtheorem*{problem}{Problem}
\newtheorem{remark}{Remark}
\newtheorem{claim}{Claim}

\newcommand{\N}{\mathbb{N}}
\newcommand{\Q}{\mathbb{Q}}
\newcommand{\R}{\mathbb{R}}
\newcommand{\C}{\mathbb{C}}

\newcommand{\ran}{\mathrm{ran}}
\newcommand{\dom}{\mathrm{dom}}
\newcommand{\id}{\mathrm{id}}
\newcommand{\med}{\mathrm{med}}

\newcommand{\bfu}{\mathbf{u}}
\newcommand{\bfv}{\mathbf{v}}
\newcommand{\bfw}{\mathbf{w}}
\newcommand{\bfx}{\mathbf{x}}
\newcommand{\bfy}{\mathbf{y}}
\newcommand{\bfz}{\mathbf{z}}

\newcommand{\Ast}{\boldsymbol{\ast}}

\begin{abstract}
We investigate the barycentric associativity property for functions with indefinite arities and discuss the more general property of barycentric preassociativity, a generalization of barycentric associativity which does not involve any composition of functions. We also provide a generalization of Kolmogoroff-Nagumo's characterization of the quasi-arithmetic mean functions to barycentrically preassociative functions.
\end{abstract}



\maketitle

\section{Introduction}

Let $X$ and $Y$ be arbitrary nonempty sets. Throughout this paper we regard tuples $\bfx$ in $X^n$ as $n$-strings over $X$. The $0$-string or \emph{empty} string is denoted by $\varepsilon$ so that $X^0=\{\varepsilon\}$. We denote by $X^*$ the set of all strings over $X$, that is, $X^*=\bigcup_{n\geqslant 0}X^n$, and we denote its elements by bold roman letters $\bfx$, $\bfy$, $\bfz$, $\ldots$ For $1$-strings, we often use non-bold italic letters $x$, $y$, $z$, $\ldots$

We endow the set $X^*$ with concatenation for which we use the juxtaposition notation. For instance, if $\bfx\in X^m$ and $y\in X$, then $\bfx y\in X^{m+1}$. Moreover, for every string $\bfx$ and every integer $n\geqslant 0$, the power $\bfx^n$ stands for the string obtained by concatenating $n$ copies of $\bfx$. In particular we have $\bfx^0=\varepsilon$. The \emph{length} of a string $\bfx$ is denoted by $|\bfx|$. For instance, $|\varepsilon|=0$.

As usual, a function $F\colon X^n\to Y$ (an \emph{operation}, if $Y=X$) is said to be \emph{$n$-ary}. Similarly, we say that a function $F\colon X^*\to Y$ has an \emph{indefinite arity} or is \emph{$\Ast$-ary}. For every integer $n\geqslant 0$, the \emph{$n$-ary part} $F_n$ of a function $F\colon X^*\to Y$ is the restriction of $F$ to $X^n$, that is, $F_n=F|_{X^n}$. The \emph{default value} of $F$ is the value given by its nullary part $F_0(\varepsilon)$. Finally, a \emph{$\Ast$-ary operation} on $X$ (or an \emph{operation} for short) is a function $F\colon X^*\to X\cup\{\varepsilon\}$, and such an operation is said to be \emph{$\varepsilon$-standard} \cite{MarTeh2} if it satisfies the condition
$$
F(\bfx)=\varepsilon \quad\Leftrightarrow\quad \bfx=\varepsilon.
$$

Recall that a $\Ast$-ary operation $F\colon X^*\to X\cup\{\varepsilon\}$ is said to be \emph{associative} (see, e.g., \cite{MarTeh,MarTehb,LehMarTeh}) if it satisfies the equation
$$
F(\bfx\bfy\bfz) ~=~ F(\bfx F(\bfy)\bfz),\qquad \bfx\bfy\bfz\in X^*.
$$

Thus defined, associativity expresses that the function value of a string does not change when replacing any of its substring with its corresponding value. For instance, the sum function over the set of real numbers, regarded as the $\varepsilon$-standard operation $F\colon\R^*\to\R\cup\{\varepsilon\}$ defined as $F_n(\bfx)=\sum_{i=1}^nx_i$ for every integer $n\geqslant 1$, is associative.

In this paper we are first interested in the following variant of associativity, called barycentric associativity.

\begin{definition}\label{de:Bassoc}
A $\Ast$-ary operation $F\colon X^*\to X\cup\{\varepsilon\}$ is said to be \emph{barycentrically associative} (or \emph{B-associative} for short) if it satisfies the identity $F(\bfx\bfy\bfz) = F(\bfx F(\bfy)^{|\bfy|}\bfz)$ for every $\bfx\bfy\bfz\in X^*$.
\end{definition}

\begin{remark}\label{rem:Fe1}
We observe from Definition~\ref{de:Bassoc} that, if $F(\bfx)\in X$ for every $\bfx\in X^*\setminus\{\varepsilon\}$ (which holds, e.g., if $F$ is $\varepsilon$-standard), then the default value $F_0(\varepsilon)$ of a B-associative operation $F\colon X^*\to X\cup\{\varepsilon\}$ is unimportant in the sense that if we modify this value, then the resulting operation is still B-associative. However, if $F(\mathbf{y})=\varepsilon$ for some $\mathbf{y}\neq\varepsilon$, then the default value $F_0(\varepsilon)$ must be $\varepsilon$. Indeed, we then have
$$
\varepsilon ~=~ F(\mathbf{y}) ~=~ F(F(\mathbf{y})^{|\mathbf{y}|}) ~=~ F(\varepsilon^{|\mathbf{y}|}) ~=~ F(\varepsilon).
$$
To give a nonconstant example of such an operation, set $a\in X$ and consider the operation $F_a\colon X^*\to X\cup\{\varepsilon\}$ defined as $F_a(\bfx)=a$, if $\bfx=\mathbf{u}a\mathbf{v}$ for some $\mathbf{u}\mathbf{v}\in X^*$, and $F_a(\bfx)=\varepsilon$, otherwise. Then $F_a$ is both associative and B-associative.
\end{remark}

By definition, B-associativity expresses that the function value of a string does not change when replacing every letter of a substring with the value of this substring. For instance, the arithmetic mean over the set of real numbers, regarded as the $\varepsilon$-standard operation $F\colon\R^*\to\R\cup\{\varepsilon\}$ defined as $F_n(\bfx)=\frac{1}{n}\sum_{i=1}^nx_i$ for every integer $n\geqslant 1$, is B-associative. However, this operation is not associative. Actually, contrary to associativity, B-associativity is satisfied by various mean functions when regarded as $\varepsilon$-standard operations over the reals, including the arithmetic mean, the geometric mean, and the harmonic mean.

To our knowledge, a simple form of B-associativity was introduced first in 1909 by Schimmack \cite{Sch09} as a natural and suitable variant of associativity to characterize the arithmetic mean over the reals. More precisely, Schimmack considered the condition $F(\bfy z)=F(F(\bfy)^{|\bfy|}z)$ for symmetric functions $F\colon\bigcup_{n\geqslant 1}\R^n\to\R$ (`symmetric' means that every $n$-ary part of $F$ is invariant under any permutation of its arguments).

A similar condition, namely $F(\bfy\bfz)=F(F(\bfy)^{|\bfy|}\bfz)$ with $|\bfz|\geqslant 1$, was then used for symmetric functions $F\colon\bigcup_{n\geqslant 1}\R^n\to\R$ in 1930 by Kolmogoroff \cite{Kol30} and independently by Nagumo \cite{Nag30} to characterize the class of quasi-arithmetic mean functions (see Theorem~\ref{thm:KolNag30} below).

The general nonsymmetric definition given in Definition~\ref{de:Bassoc} appeared more recently in \cite{Ant98} and \cite{Mar98} (see also \cite{MarMatTou99}) and both the symmetric and nonsymmetric versions of this definition have then been used to characterize further classes of functions; see, e.g.,  \cite{FodMar97,FodMar06,Mar00,MarMatTom,MarMatTou99}. For general background on B-associativity and its links with associativity, see \cite[Sect.~2.3]{GraMarMesPap09}.

Since their introduction, this condition and its different versions were used under at least three different names: \emph{associativity of means} \cite{deF31}, \emph{decomposability} \cite[Sect.~5.3]{FodRou94}, and \emph{barycentric associativity} \cite{Ant98}. Here we have chosen the third one, which naturally recalls the associativity property of the barycenter (see Remark~\ref{rem:barycenter} below).

In \cite{MarTeh,MarTehb} the authors recently introduced a generalization of associativity for $\Ast$-ary functions called preassociativity (see also \cite{LehMarTeh,MarTeh2}). A function $F\colon X^*\to Y$ is said to be \emph{preassociative} if
$$
F(\bfy)=F(\bfy')\quad\Rightarrow\quad F(\bfx\bfy\bfz)=F(\bfx\bfy'\bfz),\qquad \bfx\bfy\bfy'\bfz\in X^*.
$$

In this paper we investigate the following simultaneous generalization of preassociativity and barycentric associativity, which we call barycentric preassociativity.

\begin{definition}\label{de:BPA}
We say that a function $F\colon X^*\to Y$ is \emph{barycentrically preassociative} (or \emph{B-preassociative} for short) if for every $\bfx\bfy\bfy'\bfz\in X^*$ such that $|\bfy|=|\bfy'|$ we have
$$
F(\bfy)=F(\bfy')\quad\Rightarrow\quad F(\bfx\bfy\bfz)=F(\bfx\bfy'\bfz).
$$
\end{definition}

\begin{remark}
We observe that if we modify the default value of a B-preassociative function, then the resulting $\Ast$-ary function is still B-preassociative.
\end{remark}

Thus, a function $F\colon X^*\to Y$ is B-preassociative if the equality of the function values of two strings of the same length still holds when adding identical arguments on the left or on the right of these strings. For instance, the $\varepsilon$-standard sum operation $F\colon\R^*\to\R\cup\{\varepsilon\}$ defined as $F_n(\bfx)=\sum_{i=1}^nx_i$ for every integer $n\geqslant 1$ is B-preassociative. However, this operation is not B-associative.

By definition, B-preassociativity generalizes preassociativity. It was shown in \cite{LehMarTeh,MarTeh,MarTehb} that preassociativity generalizes associativity. Similarly, we show in this paper (Proposition~\ref{prop:BA-PBA1}) that B-preassociativity generalizes B-associativity.

B-preassociativity may be very natural in various areas. In decision making for instance, in a sense it says that if we express an indifference when comparing two profiles, then this indifference is preserved when adding identical pieces of information to these profiles. In descriptive statistics and aggregation function theory, it says that the aggregated value of a series of numerical values remains unchanged when modifying a bundle of these values without changing their partial aggregation.

B-preassociativity is not really a new property. A slightly different version was actually introduced in 1931 by de Finetti \cite[p.~380]{deF31} for symmetric functions $F\colon\bigcup_{n\geqslant 1}\R^n\to\R$. According to de Finetti, a mean function $F\colon\bigcup_{n\geqslant 1}\R^n\to\R$ is said to be `associative' if for every $x\bfy\bfz\in\bigcup_{n\geqslant 1}\R^n$, with $|\bfz|\geqslant 1$, we have $F(\bfy\bfz)=F(x^{|\bfy|}\bfz)$ whenever $F(\bfy)=F(x^{|\bfy|})$.

It is noteworthy that, contrary to B-associativity, B-preassociativity does not involve any composition of functions and hence allows us to consider a codomain $Y$ that may differ from the set $X\cup\{\varepsilon\}$. For instance, the length function $F\colon X^*\to\R$, defined as $F(\bfx)=|\bfx|$, is B-preassociative.

The outline of this paper is as follows. After going through some preliminaries in Section 2, we establish a number of important properties of B-associative and B-preassociative functions in Sections 3 and 4, respectively. In Section 4 we mainly focus on those B-preassociative functions $F\colon X^*\to Y$ for which, for every integer $n\geqslant 1$, the $n$-ary function $F_n\colon X^n\to Y$ has the same range as its diagonal section $x\mapsto F_n(x^n)$. (When $Y=X\cup\{\varepsilon\}$, these B-preassociative functions include the B-associative ones). In particular, we give a characterization of these functions as compositions of the form $F_n=f_n\circ H_n$, where $H\colon X^*\to X\cup\{\varepsilon\}$ is a B-associative $\varepsilon$-standard operation and $f_n\colon H_n(X^n)\to Y$ is one-to-one (Theorem~\ref{thm:FactoriAWRI-BPA237111}). From this result we derive a generalization of Kolmogoroff-Nagumo's characterization of the quasi-arithmetic mean functions to barycentrically preassociative functions (Theorem~\ref{thm:KolmExt}).

The terminology used throughout this paper is the following. We denote by $\N$ the set $\{1,2,3,\ldots\}$ of strictly positive integers. The domain and range of any function $f$ are denoted by $\dom(f)$ and $\ran(f)$, respectively. The identity operation on any nonempty set is denoted by $\id$. For every integer $n\geqslant 1$, the diagonal section $\delta_F\colon X\to Y$ of a function $F\colon X^n\to Y$ is defined as $\delta_F(x)=F(x^n)$. For every function $F\colon X^n\to Y$, with $n\geqslant 2$, we also let $\delta_{F}^{r}\colon X^2\to Y$ and $\delta_{F}^{\ell}\colon X^2\to Y$ be the binary functions defined as
$$
\delta_{F}^{r}(xy) ~=~ F(x^{n-1}y)\quad\mbox{and}\quad\delta_{F}^{\ell}(xy) ~=~ F(xy^{n-1}),
$$
respectively. By extension, we define $\delta_{F_0}=\delta_{F_0}^{r}=\delta_{F_0}^{\ell}=F_0$ and $\delta_{F_1}^{r}=\delta_{F_1}^{\ell}=F_1$.

\begin{remark}[Geometric interpretation of B-associativity]\label{rem:barycenter}
Consider a set of identical homogeneous balls in $X=\R^n$. Each ball is identified by the coordinates $x\in X$ of its center. Let $F\colon X^*\to X\cup\{\varepsilon\}$ be the $\varepsilon$-standard operation which carries any set of balls into their barycenter. Due to the well-known associativity-like property of the barycenter, the operation $F$ must satisfy the equation $F(\bfx\bfy\bfz)=F(\bfx F(\bfy)^{|\bfy|}\bfz)$ for every $\bfx\bfy\bfz\in X^*$ and therefore is B-associative (see Figure~\ref{fig:BA}).
\setlength{\unitlength}{6ex}
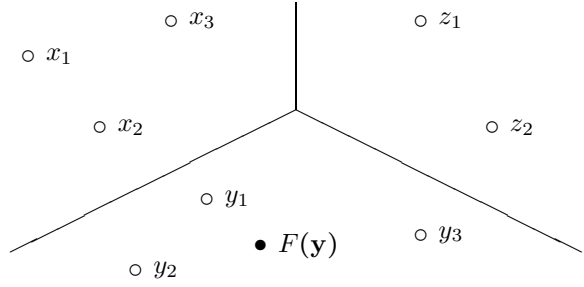
\begin{figure}[tbp]\centering
\begin{picture}(8,4)
\put(4,2.5){\line(-2,-1){4}}\put(4,2.5){\line(2,-1){4}}\put(4,2.5){\line(0,1){1.5}}
\put(1.75,0.25){\circle{0.15}}\put(2.00,0.25){\makebox(0,0)[l]{$y_2$}}
\put(2.75,1.25){\circle{0.15}}\put(3.00,1.25){\makebox(0,0)[l]{$y_1$}}
\put(5.75,0.75){\circle{0.15}}\put(6.00,0.75){\makebox(0,0)[l]{$y_3$}}
\put(3.50,0.60){\circle*{0.15}}\put(3.75,0.60){\makebox(0,0)[l]{$F(\mathbf{y})$}}
\put(0.25,3.25){\circle{0.15}}\put(0.50,3.25){\makebox(0,0)[l]{$x_1$}}
\put(1.25,2.25){\circle{0.15}}\put(1.50,2.25){\makebox(0,0)[l]{$x_2$}}
\put(2.25,3.75){\circle{0.15}}\put(2.50,3.75){\makebox(0,0)[l]{$x_3$}}
\put(5.75,3.75){\circle{0.15}}\put(6.00,3.75){\makebox(0,0)[l]{$z_1$}}
\put(6.75,2.25){\circle{0.15}}\put(7.00,2.25){\makebox(0,0)[l]{$z_2$}}
\end{picture}
\caption{Barycentric associativity} \label{fig:BA}
\end{figure}
\end{remark}

\section{Preliminaries}

Recall that, for any $n\in\N$, an $n$-ary operation $F\colon X^n\to X$ is said to be \emph{idempotent} (see, e.g., \cite{GraMarMesPap09}) if $\delta_F=\id$. An operation $F\colon X^n\to X$ is said to be \emph{range-idempotent} \cite{GraMarMesPap09} if $\delta_F|_{\ran(F)}=\id|_{\ran(F)}$, or equivalently, $\delta_F\circ F=F$. In this case $\delta_F$ necessarily satisfies the equation $\delta_F\circ\delta_F=\delta_F$.

We now introduce the following definitions. We say that a $\Ast$-ary operation $F\colon X^*\to X\cup\{\varepsilon\}$ is
\begin{itemize}
\item \emph{idempotent} if $\delta_{F_n}=\id$ for every $n\in\N$;

\item \emph{arity-wise range-idempotent} if $F(F(\bfx)^{|\bfx|})=F(\bfx)$ for every $\bfx\in X^*$ (if $F$ is $\varepsilon$-standard, this condition is equivalent to $\delta_{F_n}\circ F_n=F_n$ for every $n\in\N$).
\end{itemize}
We say that an $n$-ary function $F\colon X^n\to Y$ ($n\in\N$) is \emph{quasi-range-idempotent} if $\ran(\delta_F)=\ran(F)$ and we say that a $\Ast$-ary function $F\colon X^*\to Y$ is \emph{arity-wise quasi-range-idempotent} if $F_n$ is quasi-range-idempotent for every $n\in\N$.

We immediately observe that range-idempotent operations $F\colon X^n\to X$ are necessarily quasi-range-idempotent. The following proposition states a finer result.

\begin{proposition}
For any $n\in\N$, an operation $F\colon X^n\to X$ is range-idempotent if and only if it is quasi-range-idempotent and satisfies $\delta_F\circ \delta_F=\delta_F$.
\end{proposition}

\begin{proof}
(Necessity) We have $\ran(\delta_F)\subseteq\ran(F)$ for any operation $F\colon X^n\to X$. Since $F$ is range-idempotent, we have $\delta_F\circ F=F$, from which the converse inclusion follows immediately. In particular, $\delta_F\circ \delta_F=\delta_F$.

(Sufficiency) Since $F$ is quasi-range-idempotent, the identity $\delta_F\circ \delta_F=\delta_F$ is equivalent to $\delta_F\circ F=F$.
\end{proof}

We now show that any quasi-range-idempotent function $F\colon X^n\to Y$ ($n\in\N$) can always be factorized as $F=\delta_F\circ H$, where $H\colon X^n\to X$ is range-idempotent.

First recall that a function $g$ is a \emph{quasi-inverse} \cite[Sect.~2.1]{SchSkl83} of a function $f$ if $f\circ g|_{\ran(f)}=\id|_{\ran(f)}$ and $\ran(g|_{\ran(f)})=\ran(g)$. We then have $\ran(g)\subseteq\dom(f)$ and the function $f|_{\ran(g)}$ is one-to-one. Recall also that the Axiom of Choice (AC) is equivalent to the statement ``every function has a quasi-inverse.'' Moreover, the relation of being quasi-inverse is symmetric, i.e., if $g$ is a quasi-inverse of $f$, then $f$ is a quasi-inverse of $g$. Throughout this paper we denote the set of all quasi-inverses of $f$ by $Q(f)$.

\begin{fact}\label{fact:fgh}
Assume AC and let $f$ and $h$ be two functions such that $\ran(h)\subseteq\ran(f)$. Then we have $f\circ g\circ h=h$ for every $g\in Q(f)$.
\end{fact}

\begin{proposition}\label{prop:QRIqi2431}
Assume AC and let $F\colon X^n\to Y$ be a quasi-range-idempotent function, where $n\in\N$. For any $g\in Q(\delta_F)$, the operation $H\colon X^n\to X$ defined as $H=g\circ F$ is a range-idempotent solution of the equation $F=\delta_F\circ H$. Moreover, the function $\delta_F|_{\ran(H)}$ is one-to-one.
\end{proposition}

\begin{proof}
Let $g\in Q(\delta_F)$ and set $H=g\circ F$. Since $\ran(\delta_F)=\ran(F)$, by Fact~\ref{fact:fgh} we have $\delta_F\circ H=\delta_F\circ g\circ F=F$. Also, $H$ is range-idempotent since $\delta_H\circ H=g\circ \delta_F\circ H=g\circ F=H$. Since $\delta_F|_{\ran(g)}$ is one-to-one and $\ran(H)\subseteq\ran(g)$, the function $\delta_F|_{\ran(H)}$ is also one-to-one.
\end{proof}

The following proposition, inspired from the investigation of Chisini means \cite{Mar10}, yields necessary and sufficient conditions for a function $F\colon X^n\to Y$ to be quasi-range-idempotent.

\begin{proposition}\label{prop:ACqriTFAE3451}
Assume AC and let $F\colon X^n\to Y$ be a function, where $n\in\N$. The following assertions are equivalent.
\begin{enumerate}
\item[(i)] $F$ is quasi-range-idempotent.

\item[(ii)] There exists an operation $H\colon X^n\to X$ such that $F=\delta_F\circ H$.

\item[(iii)] There exists an idempotent operation $H\colon X^n\to X$ and a function $f\colon X\to Y$ such that $F=f\circ H$. In this case, $f=\delta_F$.

\item[(iv)] There exists a range-idempotent operation $H\colon X^n\to X$ and a function $f\colon X\to Y$ such that $F=f\circ H$. In this case, $F=\delta_F\circ H$. Moreover, if $h=\delta_F|_{\ran(H)}$ is one-to-one, then $h^{-1}\in Q(\delta_F)$.

\item[(v)] There exists a quasi-range-idempotent operation $H\colon X^n\to X$ and a function $f\colon X\to Y$ such that $F=f\circ H$.
\end{enumerate}
In assertions (ii), (iv), and (v) we may choose $H=g\circ F$ for any $g\in Q(\delta_F)$ and $H$ is then range-idempotent. In assertion (iii) we may choose $H$ such that $\delta_H=\id$ and $H=g\circ F$ on $X^n\setminus\{x^n:x\in X\}$ for any $g\in Q(\delta_F)$.
\end{proposition}

\begin{proof}
(i) $\Rightarrow$ (ii) Follows from Proposition~\ref{prop:QRIqi2431}.

(ii) $\Rightarrow$ (iii) Modifying $\delta_H$ into $\id$ and taking $f=\delta_F$, we obtain $F=f\circ H$, where $H$ is idempotent. We then have $\delta_F=f\circ \delta_H=f\circ\id =f$.

(iii) $\Rightarrow$ (iv) The first part is trivial. Also, we have $\delta_F\circ H=f\circ\delta_H\circ H=f\circ H=F$. Now, if $h=\delta_F|_{\ran(H)}$ is one-to-one, then we have $H=h^{-1}\circ F$ and hence $\delta_F\circ h^{-1}\circ \delta_F=\delta_F\circ \delta_H=h\circ \delta_H\circ \delta_H=h\circ \delta_H=\delta_F$, which shows that $h^{-1}\in Q(\delta_F)$.

(iv) $\Rightarrow$ (v) Trivial.

(v) $\Rightarrow$ (i) We have $\ran(\delta_F)=\ran(f\circ \delta_H)=\ran(f\circ H)=\ran(F)$.

The last part follows from Proposition~\ref{prop:QRIqi2431}.
\end{proof}

\begin{remark}
The proof of implication (v) $\Rightarrow$ (i) in Proposition~\ref{prop:ACqriTFAE3451} shows that the property of quasi-range-idempotence is preserved under left composition with unary maps.
\end{remark}

For any $n\in\N$, we say that a function $F\colon X^n\to Y$ is \emph{idempotizable} (see \cite{Mar10} for a variant of this definition) if it is quasi-range-idempotent and $\delta_F$ is one-to-one. In this case the composition $\delta_F^{-1}\circ F$, from $X^n$ to $X$, is idempotent. From Proposition~\ref{prop:ACqriTFAE3451}, we immediately derive the following corollary.

\begin{corollary}
Let $F\colon X^n\to Y$ be a function, where $n\in\N$. The following assertions are equivalent.
\begin{enumerate}
\item[(i)] $F$ is idempotizable.

\item[(ii)] $\delta_F$ is a bijection from $X$ onto $\ran(F)$ and there is a unique idempotent operation $H\colon X^n\to X$, namely $H=\delta_F^{-1}\circ F$, such that $F=\delta_F\circ H$.

\item[(iii)] There exists an idempotent operation $H\colon X^n\to X$ and a bijection $f$ from $X$ onto $\ran(F)$ such that $F=f\circ H$. In this case we have $f=\delta_F$ and $H=\delta_F^{-1}\circ F$.
\end{enumerate}
\end{corollary}

In the special case when the function is range-idempotent, we have the following result.

\begin{proposition}\label{prop:sd5f7xx}
For any $n\in\N$, if $F\colon X^n\to X$ is range-idempotent and $\delta_{F}$ is one-to-one, then $F$ is idempotent.
\end{proposition}

\begin{proof}
Since $\delta_{F}\circ \delta_{F}=\delta_{F}$ we simply have $\delta_{F}=\delta_{F}^{-1}\circ \delta_{F}=\id$.
\end{proof}

\section{Barycentric associativity}

In this section we discuss the concept of B-associativity (mainly for $\varepsilon$-standard operations) and provide some results on this property which will be useful in the investigation of B-preassociative functions.

Let us first recall the characterization of the class of quasi-arithmetic mean functions given by Kolmogoroff \cite{Kol30} and Nagumo \cite{Nag30}. This result, originally stated for real functions over a closed interval $[a,b]$, was extended in \cite[Sect.~4.2]{GraMarMesPap09} to functions over an arbitrary real interval $\mathbb{I}$. The following theorem gives the characterization following Kolmogoroff (we extend the domain of the functions $F\colon\bigcup_{n\geqslant 1}\mathbb{I}^n\to\mathbb{I}$ of this characterization to $\mathbb{I}^*$, see Remark~\ref{rem:Fe1}(a)). Nagumo's characterization is the same except that the strict increasing monotonicity of each function $F_n$ is replaced with the strict internality of $F_2$ (i.e., $x<y$ implies $x<F_2(x,y)<y$). Note also that a variant and a relaxation of Kolmogoroff-Nagumo's characterization can also be found in \cite{FodMar97,FodMar06,Mar00}.

\begin{theorem}[Kolmogoroff-Nagumo]\label{thm:KolNag30}
Let $\mathbb{I}$ be a nontrivial real interval (i.e., nonempty and not a singleton), possibly unbounded. A function $F\colon\mathbb{I}^*\to\mathbb{I}$ is B-associative and, for every $n\in\N$, the $n$-ary part $F_n$ is symmetric, continuous, idempotent, and strictly increasing in each argument if and only if there exists a continuous and strictly monotonic function $f\colon\mathbb{I}\to\R$ such that
$$
F_n(\bfx) ~=~ f^{-1}\bigg(\frac{1}{n}\sum_{i=1}^nf(x_i)\bigg),\qquad n\in\N.
$$
\end{theorem}

Interestingly, the following corollary shows that the idempotence property can be removed from Kolmogoroff-Nagumo's characterization.

\begin{fact}\label{fact:BassocAWRI}
If an operation $F\colon X^*\to X\cup\{\varepsilon\}$ is B-associative, then it is arity-wise range-idempotent (take $\bfx\bfz=\varepsilon$ in Definition~\ref{de:Bassoc}).
\end{fact}

\begin{corollary}\label{cor:KolNagId}
Let $F\colon X^*\to X\cup\{\varepsilon\}$ be a B-associative operation such that $F(\bfx)\in X$ for every $\bfx\in X^*\setminus\{\varepsilon\}$ and let $n\in\N$. If $\delta_{F_n}$ is one-to-one, then $\delta_{F_n}=\id$. In particular, the idempotence property is not needed in the Kolmogoroff-Nagumo's characterization (Theorem~\ref{thm:KolNag30}).
\end{corollary}

\begin{proof}
Since $F$ is B-associative, it is arity-wise range-idempotent by Fact~\ref{fact:BassocAWRI}. Since $F$ is $\varepsilon$-standard, this means that $\delta_{F_n}\circ F_n=F_n$ for every $n\in\N$. We then conclude by Proposition~\ref{prop:sd5f7xx}.
\end{proof}

The existence of nonsymmetric B-associative operations can be illustrated by the following example, introduced in \cite[p.~81]{Mar98} (see also \cite{MarMatTou99}). For every $z\in\R$, the $\varepsilon$-standard operation $M^z\colon\R^*\to\R\cup\{\varepsilon\}$ defined as
\begin{equation}\label{eq:DeltF}
M^z_n(\bfx) ~=~ \frac{\sum_{i=1}^nz^{n-i}(1-z)^{i-1}{\,}x_i}{\sum_{i=1}^nz^{n-i}(1-z)^{i-1}}{\,},\qquad n\in\N,
\end{equation}
is B-associative. Actually, one can show \cite{MarMatTom} that any B-associative $\varepsilon$-standard operation over $\R$ whose $n$-ary part is a nonconstant linear function for every $n\in\N$ is necessarily one of the operations $M^z$ $(z\in\R)$. More generally, the class of B-associative polynomial $\varepsilon$-standard operations (i.e., such that the $n$-ary part is a polynomial function for every $n\in\N$) over an infinite commutative integral domain $D$ was also characterized in \cite{MarMatTom}. This characterization shows that, up to singular cases, the typical B-associative polynomial $\varepsilon$-standard operations are linear, that is, of the form $M^z$, where $z\in D$.

The following proposition yields alternative equivalent definitions of B-associativity. Note that an analog equivalence holds for the associativity property; see \cite{CouMar11,LehMarTeh,MarTeh,MarTehb,MarTeh2}. The equivalence between definitions (i) and (iv) was observed in \cite[Sect.~2.3]{GraMarMesPap09}.

\begin{proposition}\label{eq:BA-eqdef}
Let $F\colon X^*\to X\cup\{\varepsilon\}$ be an operation. The following assertions are equivalent:
\begin{enumerate}
\item[(i)] $F$ is B-associative.

\item[(ii)] For every $\bfx\bfy\bfz,\bfx'\bfy'\bfz'\in X^*$ such that $\bfx\bfy\bfz=\bfx'\bfy'\bfz'$ we have $F(\bfx F(\bfy)^{|\bfy|}\bfz)=F(\bfx' F(\bfy')^{|\bfy'|}\bfz')$.

\item[(iii)] For every $\bfx\bfy\bfz\in X^*$ we have $F(F(\bfx\bfy)^{|\bfx\bfy|}\bfz)=F(\bfx F(\bfy\bfz)^{|\bfy\bfz|})$.

\item[(iv)] For every $\bfx\bfy\in X^*$ we have $F(\bfx\bfy)=F(F(\bfx)^{|\bfx|}F(\bfy)^{|\bfy|})$.
\end{enumerate}
\end{proposition}

\begin{proof}
We have (i) $\Rightarrow$ (ii) $\Rightarrow$ (iii) trivially. Let us prove (iii) $\Rightarrow$ (iv). Taking $\bfy\bfz=\varepsilon$ shows that $F$ is arity-wise range-idempotent. Taking $\bfx=\varepsilon$ and then $\bfz=\varepsilon$, we obtain $F(F(\bfx)^{|\bfx|}\bfy)=F(\bfx F(\bfy)^{|\bfy|})=F(F(\bfx \bfy)^{|\bfx \bfy|})=F(\bfx \bfy)$ and therefore $F(F(\bfx)^{|\bfx|}F(\bfy)^{|\bfy|})=F(\bfx\bfy)$. Finally, let us prove that (iv) $\Rightarrow$ (i). Clearly, $F$ is arity-wise range-idempotent (take $\bfy=\varepsilon$). For every $\bfx\bfy\bfz\in X^*$ we then have
\begin{eqnarray*}
F(\bfx F(\bfy)^{|\bfy|}\bfz) &=& F(F(\bfx F(\bfy)^{|\bfy|})^{|\bfx\bfy|}F(\bfz)^{|\bfz|}) ~\stackrel{(*)}{=}~ F(F(F(\bfx)^{|\bfx|}F(\bfy)^{|\bfy|})^{|\bfx\bfy|}F(\bfz)^{|\bfz|})\\
&=& F(F(\bfx\bfy)^{|\bfx\bfy|}F(\bfz)^{|\bfz|}) ~=~ F(\bfx\bfy\bfz),
\end{eqnarray*}
where, at $(*)$, we have used (iv) and the fact that $F$ is arity-wise range-idempotent.
\end{proof}

The following proposition shows that the definition of B-associativity (Definition~\ref{de:Bassoc}) remains unchanged if we upper bound the length of the string $\bfx\bfz$ by one. Our proof makes use of the B-preassociativity property and hence will be postponed to Section~\ref{sec:BPA}.

\begin{proposition}\label{prop:AsimpEquivVersion}
An operation $F\colon X^*\to X\cup\{\varepsilon\}$ is B-associative if and only if for every $\bfx\bfy\bfz\in X^*$ such that $|\bfx\bfz|\leqslant 1$ we have $F(\bfx\bfy\bfz)=F(\bfx F(\bfy)^{|\bfy|}\bfz)$.
\end{proposition}

Proposition~\ref{prop:AsimpEquivVersion} simply states that an operation $F\colon X^*\to X\cup\{\varepsilon\}$ is B-associative if and only if it satisfies the following two conditions:
\begin{enumerate}
\item[(a)] $F(\bfy)=F(F(\bfy)^{|\bfy|})$ for every $\bfy\in X^*$ (arity-wise range-idempotence),

\item[(b)] $F(x\bfy z)=F(xF(\bfy z)^{|\bfy z|})=F(F(x\bfy)^{|x\bfy|}z)$ for every $x\bfy z\in X^*$.
\end{enumerate}
In particular, an idempotent $\varepsilon$-standard operation $F\colon X^*\to X\cup\{\varepsilon\}$ is B-associative if and only if condition (b) above holds.

Interestingly, Proposition~\ref{prop:AsimpEquivVersion} also shows how a B-associative $\varepsilon$-standard operation $F\colon X^*\to X\cup\{\varepsilon\}$ can be constructed by choosing first $F_1$, then $F_2$, and so forth. In fact, $F_1$ can be chosen arbitrarily provided that it satisfies $F_1\circ F_1=F_1$. Then, if $F_k$ is already chosen for some $k\in\N$, then $F_{k+1}$ can be chosen arbitrarily from among the solutions of the following equations
\begin{eqnarray*}
&& \delta_{F_{k+1}}\circ F_{k+1} ~=~ F_{k+1}{\,},\label{eq:1rtz78z}\\
&& F_{k+1}(x\bfy z) ~=~ F_{k+1}(xF_k(\bfy z)^k) ~=~ F_{k+1}(F_k(x\bfy)^kz),\qquad x\bfy z\in X^{k+1}.\label{eq:2rtz78z}
\end{eqnarray*}

In general, finding all the possible functions $F_{k+1}$ is not an easy task. However, from the observations above we can immediately derive the following fact.

\begin{fact}\label{fact:86dsa}
Let $F\colon X^*\to X\cup\{\varepsilon\}$ be a B-associative operation.
\begin{enumerate}
\item[(a)] If $F_k$ is symmetric for some $k\geqslant 2$, then so is $F_{k+1}$.

\item[(b)] If $F_k$ is constant for some $k\in\N$, then so is $F_{k+1}$.

\item[(c)] For any sequence $c\in X^{\N}$ and every $n\in\N$, the function $G\colon X^*\to X$ defined by $G_k=F_k$, if $k\leqslant n$, and $G_k=c_k$, if $k>n$, is B-associative.
\end{enumerate}
\end{fact}

The following proposition gives a refinement of Fact~\ref{fact:86dsa}(a).

\begin{proposition}\label{prop:2sym2}
Let $F\colon X^*\to X\cup\{\varepsilon\}$ be a B-associative operation and let $k\geqslant 2$ be a integer. If the function $\bfy\in X^k\mapsto F_{k+2}(x\bfy z)$ is symmetric for every $x,z\in X$, then so is the function $\bfy\in X^{k+1}\mapsto F_{k+3}(x\bfy z)$ for every $x,z\in X$.
\end{proposition}

\begin{proof}
Let $\bfy\in X^{k+1}$. Then there exists $\bfu\in X^{k-1\geqslant 1}$ such that $\bfy=y_1\bfu y_{k+1}$. Since $F$ is B-associative, we have
\begin{eqnarray*}
F_{k+3}(x\bfy z) &=& F_{k+3}(xy_1\bfu y_{k+1}z) ~=~ F_{k+3}(F_{k+2}(xy_1\bfu y_{k+1})^{k+2}z)\\
&=& F_{k+3}(xF_{k+2}(y_1\bfu y_{k+1}z)^{k+2}).
\end{eqnarray*}
Since this expression is symmetric on $y_1\bfu$ and $\bfu y_{k+1}$, it must be symmetric on $\bfy$.
\end{proof}

As Fact~\ref{fact:86dsa}(c) shows, if $F\colon X^*\to X\cup\{\varepsilon\}$ is B-associative, then the function $F_{k+1}$ need not be idempotent, even if $F_k$ is idempotent. To give another example, consider the idempotent $\varepsilon$-standard operation $F\colon \R^*\to\R\cup\{\varepsilon\}$ defined by $F_n(\bfx)=x_1$ for every $n\in\N$ and let $F'\colon\R^*\to\R\cup\{\varepsilon\}$ be the $\varepsilon$-standard operation defined by $F'_n(\bfx)=F_n(\bfx)$ if $n\leqslant k$ for some $k\in\N$, and $F'_n(\bfx)=\max(x_1,0)$ if $n>k$. Both operations $F$ and $F'$ are B-associative.

On the other hand, we do not know whether or not $F_k$ is idempotent whenever so is $F_{k+1}$. However, for any $k,n\in\N$ we can prove that $F_k$ is idempotent whenever so is $F_{kn}$. Indeed, this observation immediately follows from the identity $\delta_{F_{kn}}=\delta_{F_{kn}}\circ\delta_{F_{k}}$, which can be obtained by setting $\bfx=x^k$ in the equation $F(\bfx^n)=F(F(\bfx)^{kn})$.

It is a well known fact that any associative $\Ast$-ary operation $F\colon X^*\to X\cup\{\varepsilon\}$ is completely determined by its nullary, unary, and binary parts (see, e.g., \cite{LehMarTeh,MarTeh,MarTehb,MarTeh2} and the references therein). As the examples above show, this property is not satisfied by the B-associative operations.

The following proposition shows that any B-associative $\varepsilon$-standard operation $F\colon X^*\to X\cup\{\varepsilon\}$ is completely determined by either of the functions $\delta_{F_k}^{r}$ or $\delta_{F_k}^{\ell}$ for every integer $k\geqslant 0$.

\begin{proposition}\label{prop:d87sad6ads}
Let $F\colon X^*\to X\cup\{\varepsilon\}$ and $G\colon X^*\to X\cup\{\varepsilon\}$ be two B-associative $\varepsilon$-standard operations such that $\delta_{F_k}^{r}=\delta_{G_k}^{r}$ or $\delta_{F_k}^{\ell}=\delta_{G_k}^{\ell}$ for every integer $k\geqslant 0$. Then $F=G$.
\end{proposition}

\begin{proof}
For any $k\geqslant 0$, if $F_k=G_k$ and for instance $\delta_{F_{k+1}}^{r}=\delta_{G_{k+1}}^{r}$, then $F_{k+1}=G_{k+1}$. Indeed, for every $\bfx z\in X^{k+1}$, we have
$$
F(\bfx z) ~=~ F(F(\bfx)^kz) ~=~ \delta_{F_{k+1}}^{r}(F_k(\bfx)z) ~=~ \delta_{G_{k+1}}^{r}(G_k(\bfx)z) ~=~ G(G(\bfx)^kz) ~=~ G(\bfx z).
$$
The result then follows from an immediate induction.
\end{proof}

Proposition~\ref{prop:d87sad6ads} motivates the following natural and important question: Find necessary and sufficient conditions on the functions $\delta_{F_k}^{r}$ or $\delta_{F_k}^{\ell}$ ($k\in\N$) for an $\varepsilon$-standard operation $F\colon X^*\to X\cup\{\varepsilon\}$ to be B-associative. The following theorem provides an answer to this question.

\begin{theorem}\label{thm:deltaFl}
Let $\phi_1\colon X\to X$ and, for every integer $k\geqslant 2$, let $\phi_k\colon X^2\to X$ and $u_k\in\{\ell,r\}$ be given. Then there exists a B-associative $\varepsilon$-standard operation $F\colon X^*\to X\cup\{\varepsilon\}$ such that $F_1=\phi_1$ and $\delta_{F_k}^{u_k}=\phi_k$ for every integer $k\geqslant 2$ if and only if the following conditions hold:
\begin{enumerate}
\item[(a)] for every $k\in\N$, we have
\begin{equation}\label{eq:thma}
\phi_{k+1}(xy) ~=~
\begin{cases}
\phi_{k+1}(\delta_{\phi_k}(x){\,}y) & \mbox{if $u_{k+1}=r$},\\
\phi_{k+1}(x{\,}\delta_{\phi_k}(y)) & \mbox{if $u_{k+1}=\ell$},
\end{cases}
\end{equation}

\item[(b)] there exists an arity-wise range-idempotent $\varepsilon$-standard operation $G\colon X^*\to X\cup\{\varepsilon\}$ such that $G_1=\phi_1$ and, for every $k\in\N$, we have
\begin{equation}\label{eq:thmb1}
G_{k+1}(x\bfy z) ~=~
\begin{cases}
\phi_{k+1}(G_k(x\bfy)z), & \mbox{if $u_{k+1}=r$},\\
\phi_{k+1}(xG_k(\bfy z)), & \mbox{if $u_{k+1}=\ell$},
\end{cases}
\end{equation}
and
\begin{equation}\label{eq:thmb2}
G_{k+1}(x\bfy z) ~=~
\begin{cases}
\delta_{G_{k+1}}^{\ell}(xG_k(\bfy z)) & \mbox{if $u_{k+1}=r$},\\
\delta_{G_{k+1}}^{r}(G_k(x\bfy)z) & \mbox{if $u_{k+1}=\ell$}.
\end{cases}
\end{equation}
\end{enumerate}
If these conditions hold, then we can take $F=G$.
\end{theorem}

\begin{proof}
(Necessity) We take $G=F$. The result then follows immediately.

(Sufficiency) We take $F=G$. Then we have $F_1=\phi_1$ trivially. Let us show by induction on $k\in\N$ that $\delta_{F_k}^{u_k}=\phi_k$. The case $k=1$ reduces to $F_1=\phi_1$. Suppose that the result holds for any $k\geqslant 1$ and let us show that it still holds for $k+1$. Assume for instance that $u_{k+1}=r$ (the other case can be dealt with dually). We have
$$
\delta_{F_{k+1}}^{u_{k+1}}(xy) ~=~ \phi_{k+1}(\delta_{F_{k}}(x)y) ~=~ \phi_{k+1}(\delta_{\phi_k}(x)y) ~=~ \phi_{k+1}(xy),
$$
where the first equality holds by Eq.~(\ref{eq:thmb1}), the second equality by the induction hypothesis, and the third equality by Eq.~(\ref{eq:thma}).

Combining condition (b) with Proposition~\ref{prop:AsimpEquivVersion}, we then observe that $F$ is B-associative. This completes the proof of the proposition.
\end{proof}

\begin{example}
Let $\phi_1\colon \R\to \R$ and $\phi_k\colon \R^2\to \R$ be defined as $\phi_1(x)=a_1x$ with $a_1\neq 0$ and $\phi_k(xy)=a_kx+b_ky$ with $a_k\neq 0$ and $b_k\neq 0$ for every integer $k\geqslant 2$. Then there exists a B-associative $\varepsilon$-standard operation $F\colon \R^*\to \R\cup\{\varepsilon\}$ such that $F_1=\phi_1$ and $\delta_{F_k}^{r}=\phi_k$ for every integer $k\geqslant 2$ if and only if $a_1=1$ and there exists $z\in\R\setminus\{0,1\}$ such that
\begin{equation}\label{eq:akbk}
a_{k+1} ~=~ z{\,}\frac{\Delta_k^z}{\Delta_{k+1}^z}\quad\mbox{and}\quad b_{k+1} ~=~ 1-a_{k+1} ~=~ \frac{(1-z)^k}{\Delta_{k+1}^z}{\,},\qquad k\in\N,
\end{equation}
where $\Delta_k^z=\sum_{i=1}^kz^{k-i}(1-z)^{i-1}$. In this case, $F$ is precisely the operation $M^z$ defined in Eq.~(\ref{eq:DeltF}).

Let us use Theorem~\ref{thm:deltaFl} to establish this result. By Eq.~(\ref{eq:thma}) we must have $a_k+b_k=1$ for every $k\geqslant 2$. Let us now construct the $\varepsilon$-standard operation $G\colon\R^*\to\R\cup\{\varepsilon\}$. Since $\phi_1=G_1$ satisfies $\phi_1\circ\phi_1=\phi_1$, we must have $a_1=1$. Then, by Eq.~(\ref{eq:thmb1}) we must have
$$
G_k(\bfx) ~=~ \sum_{i=1}^k\bigg(\prod_{j=i+1}^ka_j\bigg){\,}b_i{\,}x_i
$$
(we have set $b_1=1$) and we observe that each $G_k$ is range-idempotent. We also observe that Eq.~(\ref{eq:thmb2}) is then equivalent to the system of equations
$$
a_{k+1}b_i ~=~ a_ib_{i-1}{\,}\bigg(1-\prod_{j=1}^{k+1}a_j\bigg),\qquad i=2,\ldots,k+1,~\mbox{and}~k\geqslant 2.
$$
For every fixed value $z\in\R\setminus\{0,1\}$ of $a_2$, this system provides a unique sequence $(a_2,a_3,\ldots)$, which is given by Eq.~(\ref{eq:akbk}).
\end{example}

\section{Barycentric preassociativity}
\label{sec:BPA}

In this section we investigate the B-preassociativity property (see Definition~\ref{de:BPA}). In particular, we give a characterization of the B-preassociative and arity-wise quasi-range-idempotent functions as compositions of the form $F_n=f_n\circ H_n$, where $H\colon X^*\to X\cup\{\varepsilon\}$ is a B-associative $\varepsilon$-standard operation and $f_n\colon \ran(H_n)\to Y$ is one-to-one (Theorem~\ref{thm:FactoriAWRI-BPA237111}). We also derive a generalization of Kolmogoroff-Nagumo's characterization of the quasi-arithmetic mean functions to barycentrically preassociative functions (Theorem~\ref{thm:KolmExt}).

Just as for B-associativity, B-preassociativity may have different equivalent forms. The following proposition gives an equivalent definition based on two equalities of values.

\begin{proposition}
A function $F\colon X^*\to Y$ is B-preassociative if and only if for every $\bfx\bfx'\bfy\bfy'\in X^*$ such that $|\bfx|=|\bfx'|$ and $|\bfy|=|\bfy'|$ we have
$$
F(\bfx)=F(\bfx')~\mbox{ and }~F(\bfy)=F(\bfy')\quad\Rightarrow\quad F(\bfx\bfy)=F(\bfx'\bfy').
$$
\end{proposition}

\begin{proof}
(Necessity) Let $\bfx\bfx'\bfy\bfy'\in X^*$ such that $|\bfx|=|\bfx'|$ and $|\bfy|=|\bfy'|$. If $F(\bfx)=F(\bfx')$ and $F(\bfy)=F(\bfy')$, then we have $F(\bfx\bfy)=F(\bfx'\bfy)=F(\bfx'\bfy')$.

(Sufficiency) Let $\bfx\bfy\bfy'\bfz\in X^*$ such that $|\bfy|=|\bfy'|$. If $F(\bfy)=F(\bfy')$, then $F(\bfx\bfy)=F(\bfx\bfy')$ and finally $F(\bfx\bfy\bfz)=F(\bfx\bfy'\bfz)$.
\end{proof}

The following result provides a simplified but equivalent definition of B-preassocia{\-}tivity (exactly as Proposition~\ref{prop:AsimpEquivVersion} did for B-associativity).

\begin{proposition}\label{prop:AsimpEquivVersionP}
A function $F\colon X^*\to Y$ is B-preassociative if and only if for every $\bfx\bfy\bfy'\bfz\in X^*$ such that $|\bfy|=|\bfy'|$ and $|\bfx\bfz|= 1$ we have
$$
F(\bfy) ~=~ F(\bfy')\quad\Rightarrow\quad F(\bfx\bfy\bfz) ~=~ F(\bfx\bfy'\bfz).
$$
\end{proposition}

\begin{proof}
(Necessity) Trivial.

(Sufficiency) Repeated applications of the stated condition obviously show that $F$ is B-preassociative.
\end{proof}

As mentioned in the introduction, B-preassociativity generalizes B-associativity. Moreover, we have the following result.

\begin{proposition}\label{prop:BA-PBA1}
An operation $F\colon X^*\to X\cup\{\varepsilon\}$ is B-associative if and only if it is B-preassociative and arity-wise range-idempotent.
\end{proposition}

\begin{proof}
(Necessity) By Fact~\ref{fact:BassocAWRI} we have that $F$ is arity-wise range-idempotent. To see that it is also B-preassociative, let $\bfx\bfy\bfy'\bfz\in X^*$ such that $|\bfy|=|\bfy'|$ and $F(\bfy)=F(\bfy')$. Then we have $F(\bfx\bfy\bfz) = F(\bfx F(\bfy)^{|\bfy|}\bfz) = F(\bfx F(\bfy')^{|\bfy'|}\bfz) = F(\bfx\bfy'\bfz)$.

(Sufficiency) Let $\bfx\bfy\bfz\in X^*$. We then have $F(\bfy)=F(F(\bfy)^{|\bfy|})$ and hence $F(\bfx\bfy\bfz)=F(\bfx F(\bfy)^{|\bfy|}\bfz)$.
\end{proof}

\begin{remark}
\begin{enumerate}
\item[(a)] From Proposition~\ref{prop:BA-PBA1} it follows that a B-preassociative and idempotent operation $F\colon X^*\to X\cup\{\varepsilon\}$ is necessarily B-associative.

\item[(b)] The $\varepsilon$-standard sum operation $F\colon\R^*\to\R\cup\{\varepsilon\}$ defined as $F_n(\bfx)=\sum_{i=1}^nx_i$ for every $n\in\N$ is an instance of B-preassociative function which is not B-associative.
\end{enumerate}
\end{remark}

We are now ready to provide a very simple proof of Proposition~\ref{prop:AsimpEquivVersion}.

\begin{proof}[Proof of Proposition~\ref{prop:AsimpEquivVersion}]
The necessity is trivial. To prove the sufficiency, let $F\colon X^*\to X\cup\{\varepsilon\}$ satisfy the stated conditions. Then $F$ is clearly arity-wise range-idempotent. To see that it is B-associative, by Proposition~\ref{prop:BA-PBA1} it suffices to show that it is B-preassociative. Let $\bfx\bfy\bfy'\bfz\in X^*$ such that $|\bfy|=|\bfy'|$ and $|\bfx\bfz|= 1$ and assume that $F(\bfy)=F(\bfy')$. Then we have $F(\bfx\bfy\bfz)=F(\bfx F(\bfy)^{|\bfy|}\bfz)=F(\bfx F(\bfy')^{|\bfy'|}\bfz)=F(\bfx\bfy'\bfz)$. The conclusion then follows from Proposition~\ref{prop:AsimpEquivVersionP}.
\end{proof}

The following corollary provides a way to construct B-associative operations from associative and arity-wise quasi-range-idempotent $\varepsilon$-standard operations.

\begin{corollary}
Assume AC. For every associative and arity-wise quasi-range-idempotent $\varepsilon$-standard operation $H\colon X^*\to X\cup\{\varepsilon\}$, any $\varepsilon$-standard operation $F\colon X^*\to X\cup\{\varepsilon\}$ such that $F_n=g_n\circ H_n$ for every $n\in\N$, where $g_n\in Q(\delta_{H_n})$, is B-associative.
\end{corollary}

\begin{proof}
For every $n\in\N$, we have $\delta_{F_n}\circ F_n=g_n\circ\delta_{H_n}\circ g_n\circ H_n=g_n\circ H_n=F_n$, which shows that $F$ is arity-wise range-idempotent. Let us now show that $F$ is B-preassociative. Let $\bfx\bfy\bfy'\bfz\in X^*$ such that $|\bfy|=|\bfy'|=k$ and $F(\bfy)=F(\bfy')$. We have $H(\bfy)=(\delta_{H_k}\circ F)(\bfy)=(\delta_{H_k}\circ F)(\bfy')=H(\bfy')$ and, since $H$ is preassociative, we have $F(\bfx\bfy\bfz)=(g_n\circ H)(\bfx\bfy\bfz)=(g_n\circ H)(\bfx\bfy'\bfz)=F(\bfx\bfy'\bfz)$. By Proposition~\ref{prop:BA-PBA1}, $F$ is B-associative.
\end{proof}

The following two propositions show how new B-preassociative functions can be constructed from given B-preassociative functions by compositions with unary maps.

\begin{proposition}[Right composition]
If $F\colon X^*\to Y$ is B-preassociative then, for every function $g\colon X'\to X$, any function $H\colon X'^*\to Y$ such that $H_n=F_n\circ(g,\ldots,g)$ for every $n\in\N$ is B-preassociative. For instance, the $\varepsilon$-standard operation $F\colon \R^*\to\R\cup\{\varepsilon\}$ defined as $F_n(\bfx)=\frac{1}{n}\sum_{i=1}^nx_i^2$ for every $n\in\N$ is B-preassociative.
\end{proposition}

\begin{proof}
For $n\in\N$, $\bfx=x_1\cdots x_n\in X'^n$, and $g\colon X'\to X$, we denote by $g(\bfx)$ the $n$-string $g(x_1)\cdots g(x_n)$.

Let $\bfx\bfy\bfy'\bfz\in X'^*$ such that $|\bfy|=|\bfy'|$ and assume that $H(\bfy)=H(\bfy')$, that is, $F(g(\bfy))=F(g(\bfy'))$. By B-preassociativity of $F$ we have $F(g(\bfx)g(\bfy)g(\bfz))=F(g(\bfx)g(\bfy')g(\bfz))$ and hence $H(\bfx\bfy\bfz)=H(\bfx\bfy'\bfz)$.
\end{proof}

\begin{proposition}[Left composition]\label{prop:leftcomp56}
Let $F\colon X^*\to Y$ be a B-preassociative function and let $(g_n)_{n\in\N}$ be a sequence of functions from $Y$ to $Y'$. If $g_n|_{\ran(F_n)}$ is one-to-one for every $n\in\N$, then any function $H\colon X^*\to Y'$ such that $H_n=g_n\circ F_n$ for every $n\in\N$ is B-preassociative. For instance, the $\varepsilon$-standard operation $F\colon \R^*\to\R\cup\{\varepsilon\}$ defined as $F_n(\bfx)=\exp(\sum_{i=1}^nx_i)$ for every $n\in\N$ is B-preassociative.
\end{proposition}

\begin{proof}
Assume that $g_n|_{\ran(F_n)}$ is one-to-one for every $n\in\N$. Then we have $F_n=f_n\circ H_n$, with $f_n=(g_n|_{\ran(F_n)})^{-1}$. Let $\bfx\bfy\bfy'\bfz\in X^*$ such that $|\bfy|=|\bfy'|=n\geqslant 1$ and assume that $H(\bfy)=H(\bfy')$. We then have $F(\bfy)=(f_n\circ H)(\bfy)=(f_n\circ H)(\bfy')=F(\bfy')$ and hence $F(\bfx\bfy\bfz)=F(\bfx\bfy'\bfz)$ by B-preassociativity of $F$. Setting $m=|\bfx\bfy\bfz|$, it follows that $H(\bfx\bfy\bfz)=(g_m\circ F)(\bfx\bfy\bfz)=(g_m\circ F)(\bfx\bfy'\bfz)=H(\bfx\bfy'\bfz)$.
\end{proof}

\begin{remark}\label{rem:s8d76}
\begin{enumerate}
\item[(a)] If $F\colon X^*\to Y$ is a B-preassociative function and $(g_n)_{n\in\N}$ is a sequence of functions from $X'$ to $X$, then any function $H\colon X'^*\to Y$ such that $H_n=F_n\circ(g_n,\ldots,g_n)$ need not be B-preassociative. For instance, consider the $\varepsilon$-standard sum operation $F_n(\bfx)=\sum_{i=1}^nx_i$ over the reals and the sequence $g_n(x)=\exp(nx)$. Then, for $x_1=\log(1)$, $x_2=\log(2)$, $x'_1=\frac{1}{2}\log(3)$, $x'_2=\frac{1}{2}\log(2)$, and $x_3=0$, we have $H(x_1x_2)=H(x'_1x'_2)$ but $H(x_1x_2x_3)\neq H(x'_1x'_2x_3)$.

\item[(b)] B-preassociativity is not always preserved by left composition of a B-pre\-asso\-ciative function with a unary map. For instance, consider the $\varepsilon$-standard sum operation $F_n(\bfx)=\sum_{i=1}^nx_i$ over the reals and let $g(x)=\max\{x,0\}$. Then for any operation $H\colon\R^*\to\R\cup\{\varepsilon\}$ such that $H_n=g\circ F_n$ for every $n\in\N$, we have $H(-1,-2)=0=H(-1,1)$ but $H(-1,-2,1)=0\neq 1=H(-1,1,1)$. Thus $H$ is not B-preassociative.
\end{enumerate}
\end{remark}

We also have the following two propositions, which generalize Fact~\ref{fact:86dsa} and Proposition~\ref{prop:2sym2}. The proofs are straightforward and thus omitted.

\begin{proposition}
Let $F\colon X^*\to Y$ be a B-preassociative function.
\begin{enumerate}
\item[(a)] If $F_k$ is symmetric for some $k\geqslant 2$, then so is $F_{k+1}$.

\item[(b)] If $F_k$ is constant for some $k\in\N$, then so is $F_{k+1}$.

\item[(c)] For any sequence $c\in Y^{\N}$ and every $n\in\N$, the function $G\colon X^*\to Y$ defined by $G_k=F_k$, if $k\leqslant n$, and $G_k=c_k$, if $k>n$, is B-preassociative.
\end{enumerate}
\end{proposition}

%
%

\begin{proposition}
Let $F\colon X^*\to Y$ be a B-preassociative function and let $k\geqslant 2$ be an integer. If the function $\bfy\in X^k\mapsto F_{k+2}(x\bfy z)$ is symmetric for every $x,z\in X$, then so is the function $\bfy\in X^{k+1}\mapsto F_{k+3}(x\bfy z)$ for every $x,z\in X$.
\end{proposition}


We now focus on those B-preassociative functions which are arity-wise quasi-range-idempotent, that is, such that $\ran(\delta_{F_n})=\ran(F_n)$ for every $n\in\N$. As we will now show, this special class of functions has very interesting and even surprising properties. First of all, just as for B-associative $\varepsilon$-standard operations, B-preassociative and arity-wise quasi-range-idempotent functions $F\colon X^*\to Y$ are completely determined by either of the functions $\delta_{F_k}^{r}$ or $\delta_{F_k}^{\ell}$ for every $k\in\N$.

\begin{proposition}\label{prop:d87sad6ads2}
Assume AC and let $F\colon X^*\to Y$ and $G\colon X^*\to Y$ be two B-preassociative and arity-wise quasi-range-idempotent functions such that $\delta_{F_k}^{r}=\delta_{G_k}^{r}$ or $\delta_{F_k}^{\ell}=\delta_{G_k}^{\ell}$ for every integer $k\geqslant 0$. Then $F=G$.
\end{proposition}

\begin{proof}
For any $k\geqslant 0$, if $F_k=G_k$ and for instance $\delta_{F_{k+1}}^{r}=\delta_{G_{k+1}}^{r}$, then $F_{k+1}=G_{k+1}$. Indeed, for every $\bfx z\in X^{k+1}$, by arity-wise quasi-range-idempotence there exists $u\in X$ such that $F_k(\bfx)=\delta_{F_k}(u)$. Since $F_k=G_k$, we also have $G_k(\bfx)=\delta_{G_k}(u)$. By B-preassociativity, we then have $F(\bfx z) = \delta_{F_{k+1}}^{r}(uz) = \delta_{G_{k+1}}^{r}(uz) = G(\bfx z)$. The result then follows from an immediate induction.
\end{proof}

We now give a characterization of the B-preassociative and arity-wise quasi-range-idempotent functions as compositions of B-associative $\varepsilon$-standard operations with one-to-one unary maps. We first consider a lemma, which provides equivalent conditions for an arity-wise quasi-range-idempotent function to be B-preassociative.

\begin{lemma}\label{lemma:UQRIrew67}
Assume AC and let $F\colon X^*\to Y$ be an arity-wise quasi-range-idempo{\-}tent function. The following assertions are equivalent.
\begin{enumerate}
\item[(i)] $F$ is B-preassociative.

\item[(ii)] For every sequence $(g_n\in Q(\delta_{F_n}))_{n\in\N}$, the $\varepsilon$-standard operation $H\colon X^*\to X\cup\{\varepsilon\}$ defined as $H_n=g_n\circ F_n$ for every $n\in\N$ is B-associative.

\item[(iii)] There is a sequence $(g_n\in Q(\delta_{F_n}))_{n\in\N}$ such that the $\varepsilon$-standard operation $H\colon X^*\to X\cup\{\varepsilon\}$ defined as $H_n=g_n\circ F_n$ for every $n\in\N$ is B-associative.
\end{enumerate}
\end{lemma}

\begin{proof}
(i) $\Rightarrow$ (ii)  By Proposition~\ref{prop:QRIqi2431}, $H$ is arity-wise range-idempotent. Since $g_n|_{\ran(\delta_{F_n})}=g_n|_{\ran(F_n)}$ is one-to-one for every $n\in\N$, by Proposition~\ref{prop:leftcomp56} the operation $H$ is B-preassociative. It follows that $H$ is B-associative by Proposition~\ref{prop:BA-PBA1}.

(ii) $\Rightarrow$ (iii) Trivial.

(iii) $\Rightarrow$ (i) By Proposition~\ref{prop:BA-PBA1} we have that $H$ is B-preassociative. For every $n\in\N$, since $g_n|_{\ran(F_n)}$ is a one-to-one function from $\ran(F_n)$ onto $\ran(g_n)$, we have $F_n = (g_n|_{\ran(F_n)})^{-1}\circ H_n$ and the function $(g_n|_{\ran(F_n)})^{-1}|_{\ran(H_n)}$ is one-to-one from $\ran(H_n)$ onto $\ran(F_n)$. By Proposition~\ref{prop:leftcomp56} it follows that $F$ is B-preassociative.
\end{proof}

\begin{remark}
Let $F\colon X^*\to Y$ be a B-preassociative function such that $F_n=f_n\circ H_n$ for every $n\in\N$, where $f_n\colon X\to Y$ is any function and $H\colon X^*\to X\cup\{\varepsilon\}$ is any arity-wise range-idempotent operation. Then $F_n=\delta_{F_n}\circ H_n$ for every $n\in\N$ by Proposition~\ref{prop:ACqriTFAE3451}(iv). However, $H$ need not be B-associative. For instance, if $F$ is a constant function, then $H$ could be any arity-wise range-idempotent function. However, Lemma~\ref{lemma:UQRIrew67} shows that, assuming AC, there is always a B-associative solution $H$ of the equation $F_n=\delta_{F_n}\circ H_n$; for instance, $H_0(\varepsilon)=\varepsilon$ and $H_n=g_n\circ F_n$ for $g_n\in Q(\delta_{F_n})$ and $n\in\N$.
\end{remark}

\begin{theorem}\label{thm:FactoriAWRI-BPA237111}
Assume AC and let $F\colon X^*\to Y$ be a function. The following assertions are equivalent.
\begin{enumerate}
\item[(i)] $F$ is B-preassociative and arity-wise quasi-range-idempotent.

\item[(ii)] There exists a B-associative $\varepsilon$-standard operation $H\colon X^*\to X\cup\{\varepsilon\}$ and a sequence $(f_n\colon\ran(H_n)\to Y)_{n\in\N}$ of one-to-one functions such that $F_n=f_n\circ H_n$ for every $n\in\N$.
\end{enumerate}
If condition (ii) holds, then for every $n\in\N$ we have $F_n=\delta_{F_n}\circ H_n$, $f_n=\delta_{F_n}|_{\ran(H_n)}$, $f_n^{-1}\in Q(\delta_{F_n})$, and we may choose $H_n=g_n\circ F_n$ for any $g_n\in Q(\delta_{F_n})$.
\end{theorem}

\begin{proof}
(i) $\Rightarrow$ (ii) Let $g_n\in Q(\delta_{F_n})$ for every $n\in\N$ and consider the $\varepsilon$-standard operation $H\colon X^*\to X\cup\{\varepsilon\}$ defined as $H_n=g_n\circ F_n$ for every $n\in\N$. By Proposition~\ref{prop:QRIqi2431}, we have $F_n=f_n\circ H_n$, where $f_n=\delta_{F_n}|_{\ran(H_n)}$ is one-to-one. By Lemma~\ref{lemma:UQRIrew67}, $H$ is B-associative.

(ii) $\Rightarrow$ (i) $F$ is arity-wise quasi-range-idempotent by Proposition~\ref{prop:ACqriTFAE3451}. It is also B-preassociative by Proposition~\ref{prop:leftcomp56}.

The last part follows from Proposition~\ref{prop:ACqriTFAE3451}(iv) and Lemma~\ref{lemma:UQRIrew67}.
\end{proof}

\begin{remark}\label{rem:fsd68s}
A function $F\colon X^*\to Y$ such that $F_n=\delta_{F_n}\circ H_n$ for every $n\in\N$, where $H$ is B-associative, need not be B-preassociative. The example given in Remark~\ref{rem:s8d76}(b) illustrates this observation. To give a second example, take $X=\R$, $F_n(\bfx)=|\frac{1}{n}\sum_{i=1}^nx_i|$ and $H_n(\bfx)=\frac{1}{n}\sum_{i=1}^nx_i$ for every $n\in\N$. Then $F(1)=F(-1)$ but $F(11)=1\neq 0=F(1(-1))$. Thus $F$ is not B-preassociative.
\end{remark}

The following two results concern B-associative functions whose $n$-ary part is idempotizable (i.e., quasi-range-idempotent with a one-to-one diagonal section) for every $n\in\N$.

\begin{proposition}\label{prop:dsa98as}
Assume AC and let $F\colon X^*\to Y$ be a function. If condition (ii) of Theorem~\ref{thm:FactoriAWRI-BPA237111} holds, then the following assertions are equivalent.
\begin{enumerate}
\item[(i)] $\delta_{F_n}$ is one-to-one for every $n\in\N$,

\item[(ii)] $\delta_{H_n}$ is one-to-one for every $n\in\N$,

\item[(iii)] $\delta_{H_n}=\id$ for every $n\in\N$.
\end{enumerate}
\end{proposition}

\begin{proof}
$(i)\Rightarrow (iii)$ $\delta_{H_n}=\delta_{F_n}^{-1}\circ \delta_{F_n}=\id$.

$(iii)\Rightarrow (ii)$ Trivial.

$(ii)\Rightarrow (i)$ $\delta_{F_n}=f_n\circ\delta_{H_n}$ is one-to-one as a composition of one-to-one functions.
\end{proof}

\begin{corollary}\label{cor:sfa5sfds}
Let $F\colon X^*\to Y$ be a function such that $\delta_{F_n}$ is one-to-one for every $n\in\N$. The following assertions are equivalent.
\begin{enumerate}
\item[(i)] $F$ is B-preassociative and arity-wise quasi-range-idempotent.

\item[(ii)] There is a B-associative and idempotent $\varepsilon$-standard operation $H\colon X^*\to X\cup\{\varepsilon\}$ such that $F_n=\delta_{F_n}\circ H_n$ for every $n\in\N$.
\end{enumerate}
\end{corollary}

\begin{proof}
Follows from Theorem~\ref{thm:FactoriAWRI-BPA237111} and Proposition~\ref{prop:dsa98as}. Here AC is not required since the quasi-inverse of $\delta_{F_n}$ is simply an inverse.
\end{proof}

Applying Corollary~\ref{cor:sfa5sfds} to the class of quasi-arithmetic mean functions (Theorem~\ref{thm:KolNag30}), we obtain the following generalization of Kolmogoroff-Nagumo's characterization.

\begin{theorem}\label{thm:KolmExt}
Let $\mathbb{I}$ be a nontrivial real interval, possibly unbounded. A function $F\colon \mathbb{I}^*\to\R$ is B-preassociative and, for every $n\in\N$, the function $F_n$ is symmetric, continuous, and strictly increasing in each argument if and only if there are continuous and strictly increasing functions $f\colon\mathbb{I}\to\R$ and $f_n\colon\R\to\R$ $(n\in\N)$ such that
$$
F_n(\bfx) ~=~ f_n\bigg(\frac{1}{n}\sum_{i=1}^n f(x_i)\bigg),\qquad n\in\N.
$$
\end{theorem}

\begin{proof}
(Necessity) Let $n\in\N$ and $y\in\ran(F_n)$. Since $F_n$ is increasing, for any $\bfx=x_1\cdots{\,}x_n\in X^n$ such that $F(\bfx)=y$ we have
$$
\delta_{F_n}(\min\{x_1,\ldots,x_n\})~\leqslant ~ y ~\leqslant ~\delta_{F_n}(\max\{x_1,\ldots,x_n\}).
$$
Since $\delta_{F_n}$ is continuous, it follows that $y\in\ran(\delta_{F_n})$. Therefore $\ran(F_n)\subseteq\ran(\delta_{F_n})$ and hence $F$ is arity-wise quasi-range-idempotent.

By Corollary~\ref{cor:sfa5sfds}, the $\varepsilon$-standard operation $H\colon\mathbb{I}^*\to\mathbb{I}\cup\{\varepsilon\}$ defined as $H_n=\delta_{F_n}^{-1}\circ F_n$ for every $n\in\N$ is B-associative and every $H_n$ is idempotent, strictly increasing in each variable, continuous, and symmetric. By Theorem~\ref{thm:KolNag30}, there is a continuous and strictly increasing function $f\colon\mathbb{I}\to\R$ such that
$$
H_n(\bfx) ~=~ f^{-1}\bigg(\frac{1}{n}\sum_{i=1}^n f(x_i)\bigg),\qquad n\in\N.
$$
To conclude, it suffices to define $f_n\colon\R\to\R$ as $f_n=\delta_{F_n}\circ f^{-1}$.

(Sufficiency) For every $n\in\N$ we clearly have $\delta_{F_n}=f_n\circ f$, that is, $f_n=\delta_{F_n}\circ f^{-1}$. It follows that $F$ is B-preassociative by Corollary~\ref{cor:sfa5sfds}. The other properties are immediate.
\end{proof}

The axiomatization given in Theorem~\ref{thm:KolmExt} enables us to introduce the following definition.

\begin{definition}
Let $\mathbb{I}$ be a nontrivial real interval, possibly unbounded. We say that a function $F\colon \mathbb{I}^*\to\R$ is a \emph{quasi-arithmetic pre-mean function} if there are continuous and strictly increasing functions $f\colon\mathbb{I}\to\R$ and $f_n\colon\R\to\R$ $(n\in\N)$ such that
$$
F_n(\bfx) ~=~ f_n\bigg(\frac{1}{n}\sum_{i=1}^n f(x_i)\bigg),\qquad n\in\N.
$$
\end{definition}

\begin{remark}
As expected, the class of quasi-arithmetic pre-mean functions includes all the quasi-arithmetic mean functions (just take $f_n=f^{-1}$). Actually the quasi-arithmetic mean functions are exactly those quasi-arithmetic pre-mean functions which are idempotent. However, there are also many non-idempotent quasi-arithmetic pre-mean functions. Taking for instance $f_n(x)=nx$ and $f(x)=x$ over the reals $\mathbb{I}=\R$, we obtain the sum function. Taking $f_n(x)=\exp(nx)$ and $f(x)=\ln(x)$ over $\mathbb{I}=\left]0,\infty\right[$, we obtain the product function.
\end{remark}

The following proposition shows that the generators $f_n$ and $f$ defined in Theorem~\ref{thm:KolmExt} are defined up to an affine transformation.

\begin{proposition}\label{prop:genfg}
Let $\mathbb{I}$ be a nontrivial real interval, possibly unbounded. Let $f,g\colon\mathbb{I}\to\R$ and $f_n,g_n\colon\R\to\R$ $(n\in\N)$ be continuous and strictly monotonic functions. Then the functions $f_n(\frac{1}{n}\sum_{i=1}^n f(x_i))$ and $g_n(\frac{1}{n}\sum_{i=1}^n g(x_i))$ coincide on $\mathbb{I}^n$ if and only if there exist $r,s\in\R$, $r\neq 0$, such that $g_n^{-1}\circ f_n=g\circ f^{-1}=r{\,}\id+s$ for every $n\in\N$.
\end{proposition}

\begin{proof}
Let us prove the necessity. Setting $z_i=f(x_i)$, we see that the mentioned functions coincide on $\mathbb{I}^n$ if and only if
$$
(g_n^{-1}\circ f_n)\bigg(\frac{1}{n}\sum_{i=1}^n z_i\bigg) ~=~ \frac{1}{n}\sum_{i=1}^n (g\circ f^{-1})(z_i),\qquad n\in\N.
$$
Identifying the variables in this identity yields $g_n^{-1}\circ f_n=g\circ f^{-1}$ for every $n\in\N$. It follows that the continuous function $h=g\circ f^{-1}$ satisfies the Jensen equality $h(\frac{1}{n}\sum_{i=1}^n z_i)=\frac{1}{n}\sum_{i=1}^n h(z_i)$. Therefore there exist $r,s\in\R$, $r\neq 0$, such that $h(x)=rx+s$ (see \cite[p.~48]{Acz66}). The sufficiency is obvious.
\end{proof}

%

\section{Concluding remarks and open problems}

We have investigated the B-associativity for $\Ast$-ary operations as well as a relaxation of this property, namely B-preassociativity. In particular, we have presented a characterization of those B-preassociative functions which are arity-wise quasi-range-idempotent.

We end this paper with the following questions:
\begin{enumerate}
\item[(a)] Prove or disprove: If an operation $F\colon X^*\to X\cup\{\varepsilon\}$ is B-associative, then there exists a B-associative and idempotent operation $G\colon X^*\to X\cup\{\varepsilon\}$ such that $F_n=\delta_{F_n}\circ G_n$ for every $n\in\N$.

\item[(b)] Prove or disprove: Let $F\colon X^*\to X\cup\{\varepsilon\}$ be a B-associative operation. If $F_{k+1}$ is idempotent for some $k\in\N$, then so is $F_k$.

\item[(c)] Find a generalization of Theorem~\ref{thm:FactoriAWRI-BPA237111} by removing the arity-wise quasi-range-idempo{\-}tence property.

\item[(d)] Find necessary and sufficient conditions on $\delta_{F_n}$ ($n\in\N$) for a function $F\colon X^*\to Y$ satisfying $F_n=\delta_{F_n}\circ H_n$, where $H$ is B-associative, to be B-preassociative (cf.\ Remark~\ref{rem:fsd68s}).

\item[(e)] Find a characterization of those quasi-arithmetic pre-mean functions which are preassociative.
\end{enumerate}

\section*{Acknowledgments}

This research is supported by the internal research project F1R-MTH-PUL-12RDO2 of the University of Luxembourg.


\end{document}